\documentclass[12pt]{amsart}

\usepackage[margin=1in]{geometry}

\usepackage{float}

\usepackage[toc,page]{appendix}
\usepackage{datetime}
\usepackage[T1]{fontenc}
\usepackage{chngcntr}
\usepackage[T1]{fontenc}
\usepackage{amssymb,url,xspace,amsthm,adjustbox}
\usepackage[alphabetic,nobysame]{amsrefs}
\usepackage{mathtools}
\usepackage{graphicx}
\usepackage{amssymb}
\usepackage{mathrsfs,stmaryrd}
\usepackage{yfonts, bbm}
\usepackage{enumitem}

\newcommand\norm[1]{\lVert#1\rVert}

\usepackage{hyperref}
\hypersetup{
	colorlinks   = true, 
	urlcolor     = blue, 
	linkcolor    = blue, 
	citecolor   = green 
}

\usepackage{tikz}
\usetikzlibrary{shapes,arrows,calc,matrix}
\usepackage{tikz-cd}
\usepackage{todonotes}
\usepackage{color}

\usepackage{amsmath}
\usepackage{lipsum}
\usepackage{setspace}

\theoremstyle{plain}
\newtheorem{theorem}{Theorem}

\newtheorem{lemma}[theorem]{Lemma}
\newtheorem{corollary}[theorem]{Corollary}

\theoremstyle{definition}
\newtheorem{definition}{Definition}

\newtheorem*{conjecture*}{Conjecture}

\usepackage{scalerel, stackengine}
\stackMath
\newcommand\rwh[1]{%
	\savestack{\tmpbox}{\stretchto{%
			\scaleto{%
				\scalerel*[\widthof{\ensuremath{#1}}]{\kern-.6pt\bigwedge\kern-.6pt}%
				{\rule[-\textheight/2]{1ex}{\textheight}}
			}{\textheight}%
		}{0.5ex}}%
	\stackon[1pt]{#1}{\tmpbox}%
}

\newcommand{\mb}{\mathbb}
\newcommand{\mf}{\mathfrak}

\newcommand{\lb}{\left(}
\newcommand{\rb}{\right)}

\newcommand{\bc}{\mathbb{C}}

\newcommand{\bq}{\mathbb{Q}}

\newcommand{\bz}{\mathbb{Z}}

\newcommand{\tx}{\text}

\newcommand{\Ad}{\mathrm{Ad}}

\newcommand{\Rep}{\mathrm{Rep}}

\newcommand{\GL}{\mathrm{GL}}
\newcommand{\SL}{\mathrm{SL}}

\makeindex

\title{A short note on Arthur parameters with singular orbit closures}

\author{Kristaps John Balodis}

\date{\today}

\begin{document}
	
	\maketitle

	\begin{abstract}
		This short article resolves a question about Arthur parameters highlighted in \cite[Section 2.2]{BCFZ}, in the case of $\GL_n/F$ where $F/\bq_p$ is a $p$-adic field.  
		Namely, we prove that every Arthur parameter $\psi$ of $\GL_n/F$, can be factored into Arthur parameters $\psi\cong \psi_1\oplus\cdots \oplus\psi_r$ such that the orbit $C_\psi=C_{\psi_1}\times\cdots\times  C_{\psi_r}\subseteq V_{\lambda_{\psi_1}}\times \cdots \times V_{\lambda_{\psi_r}}$ in the associated Vogan variety has smooth closure if and only if each $\bar{C}_{\psi_i}$ is smooth for $i=1, \ldots, r$, and that $C_{\psi_i}$ has smooth closure if and only if it is open or closed. 
	\end{abstract}
	
	Let $F/\bq_p$ be a $p$-adic field and $G$ be a reductive group defined over $F$.  
	An \emph{infinitesimal parameter} $\lambda$ of $G$ is a continuous Frobenius semisimple homomorphism $\lambda:W_F\to {}^LG$ from the Weil group $W_F$ to the $L$-group ${}^LG$.  
	Every Langlands parameter $\phi:W_F\times\SL_2(\bc)\to {}^LG$ determines an infinitesimal parameter 
	$$\lambda_\phi(w):=\phi\lb w, \begin{pmatrix}
		|w|^{1/2} & 0 \\
		0 & |w|^{-1/2}
	\end{pmatrix}\rb,$$
	where $\norm{\cdot}$ is the norm on $W_F$.  
	To each infinitesimal parameter $\lambda$, \cite{Vog} associated a $\bc$-variety $V_\lambda$ and a group $H_\lambda$ acting on $V_\lambda$ in such a way that the $H_\lambda$-orbits are in bijection with the equivalence classes of Langlands parameters $\phi$ such that $\lambda=\lambda_\phi$. 
	Given $\phi$, we write $C_\phi$ for the orbit associated to $\phi$.  
	Likewise, each Arthur parameter 
	$$\psi:W_F\times\SL_2(\bc)\times\SL_2(\bc)\to {}^LG$$
	determines a Langlands parameter 
	$$\phi_\psi(w, x):=\psi\lb w, x, \begin{pmatrix}
		|w|^{1/2} & 0 \\
		0 & |w|^{-1/2}
	\end{pmatrix}\rb,$$
	and thus an orbit $C_\psi:=C_{\phi_\psi}$ in the Vogan variety of $\lambda_\psi:=\lambda_{\phi_\psi}$. 
	
	In \cite[Section 2.2]{BCFZ} it was demonstrated in many cases that for an Arthur parameter $\psi$ of $\GL_n$ or a classical group that $C_\psi$ is open, closed, or has singular closure.  
	Notably, it was pointed out in \cite[Example 2.10]{BCFZ} that this fails for $G_2$, but only for exactly two Arthur parameters, having equivalent infinitesimal parameter $\lambda$.    
	In this case, $H_\lambda=\mb{G}_m^2$ acts on $V_\lambda=\mb{A}_\bc^2$ by $(s, t)(x, y)=(sx, st^2y)$.  
	For the surjection
	\begin{align*}
		H_\lambda &\xrightarrow{\varphi} \mb{G}_m^2\\
		(s, t)&\mapsto(s, st^2)
	\end{align*}
	and for the identity isomorphism $\tx{Id}:V_\lambda\to \mb{A}_\bc^2$ of varieties,  $\tx{Id}(gx)=\varphi(g)\tx{Id}(x)$ for all $x\in V_\lambda$ and all $g\in H_\lambda$.  
	This motivates us to introduce the following terminology.  
	
	\begin{definition}
	A  $G$-variety $X$ is said to be \emph{factorizable} if there exists an $H_i$-variety $Y_i$ for $i=1, 2$, such that neither $Y_i$ is a point, a surjective map $\varphi:G\to H_1\times H_2$ and an isomorphism $f:X\to Y_1\times Y_2$ such that for all $g\in G$ and all $x\in X$,
	$$f(gx)=\varphi(g)f(x).$$
	\end{definition}
	
	Note that we do not require $\varphi$ to be an isomorphism, since if we have an $H$-variety $Y$, and a surjection $\varphi:G\to H$, then the $G$-action on $Y$ given by composing with $\varphi$ has the same orbits as $H$ on $Y$. 
	
	We now reformulate the hypothesis as:  If $\psi$ is an Arthur parameter of a split reductive group such that $V_{\lambda_\psi}$ is non-factorizable, then $C_\psi$ is open, closed, or has singular closure.   
	We will prove this for $\GL_n$ in Theorem \ref{thm}, and in fact, we will provide even more refined information about exactly when Arthur parameters of $\GL_n$ have orbits of smooth closure.  

	From the discussion in \cite[Sections 3.4, 3.5]{Bal} for every infinitesimal parameter $\lambda$ of $\GL_n(F)$, there exists finite extensions $E_1, \ldots, E_r$ of F, integers $m_1, \ldots, m_r\leq n$ and unramified infinitesimal parameters $\lambda_i$ of $\GL_{m_i}(E_i)$ for $i=1, \ldots, r$ such that there is an equivariant isomorphism $V_{\lambda}\cong V_{\lambda_1}\times\cdots \times V_{\lambda_r}, H_{\lambda}\cong H_{\lambda_1}\times\cdots \times H_{\lambda_r}$.  
	Moreover, \cite{Bal} describes the precise matching between orbits in $V_{\lambda}$ and orbits in $ V_{\lambda_1}\times\cdots \times V_{\lambda_r}$,  using Zelevinsky's multisegments.  
	Together with the description of multisegments of Arthur parameters in \cite{CR24}, one can see that for an Arthur parameter $\psi$ with $\lambda_\psi$ equivalent to $\lambda$, there exist unramified Arthur parameters $\psi_1, \ldots, \psi_r$ with $\lambda_{\psi_i}$ equivalent to $\lambda_i$, and that $C_\psi$ gets sent to $C_{\psi_1}\times\cdots \times C_{\psi_r}$ under the above isomorphism.  
	A special case of this was mentioned in the proof of \cite[Theorem 6.4]{CR24}.

	Thus, for $\GL_n(F)$, it suffices to prove the results of this article in the case where $\lambda$ is \emph{unramified}, meaning it is trivial on the inertia subgroup of $W_F$.  
	Let $q$ be the cardinality of the residue field of $F$, and fix a choice of Frobenius element $\mf{frob}\in W_F$ such that $\norm{\mf{frob}}=q$.  
	Then, for unramified $\lambda$,
	 \begin{align*}
	 	V_\lambda&:=\{x\in \mf{gl}_n(\bc) | \  \Ad(\lambda(\mf{frob}))x=qx\},\\
	 	H_\lambda&:=\{g\in \GL_n(\bc) | \ \lambda(\mf{frob})g=g\lambda(\mf{frob})\}.
	 \end{align*}
		
	Viewed as a representation of $W_F\times \SL_2(\bc)\times\SL_2(\bc)$, the irreducible unramified Arthur parameters of $\GL_n(F)$ are of the form
	$$\psi(w, x, y)=\norm{\cdot}^{bi}\boxtimes S^d(x)\boxtimes S^a(y),$$
	for the norm $\norm{\cdot}$ on $W_F$, $b\in [0, 2\pi/\log q )$, and where $S^n$ denotes the $(n+1)$-dimensional irreducible representation $\tx{Sym}^n(\bc^2)$ of $\SL_2(\bc)$.  
	Every unramified Arthur parameter of $\GL_n$ is given by a direct sum of irreducible unramified Arthur parameters. 

	By \cite[Equation (18)]{CR24} one can see that if $\psi$ is an irreducible unramified Arthur parameter of the form $\psi=1\boxtimes S^d\boxtimes S^a$, then the eigenvalues of $\lambda_\psi(\mf{frob})$ are all of the form $q^n$ for some $n\in \bz$, or all of the form $q^{(2n+1)/2}$ for some $n\in \bz$.  
	Hence, for $\psi=\norm{\cdot}^{ib}\boxtimes S^d\boxtimes S^a$, the eigenvalues of $\lambda_\psi(\mf{frob})$ are all of the form $q^{ib}q^n$ for some $n\in \bz$, or all of the form $q^{ib}q^{(2n+1)/2}$ for some $n\in \bz$.  
	Therefore, we can decompose an arbitrary unramified Arthur parameter $\psi\cong \psi_1^e\oplus\psi_1^o\oplus \cdots \oplus \psi_r^e\oplus \psi_r^o$ such that for some distinct $b_1, \ldots, b_r\in [0, 2\pi /\log q)$, the eigenvalues of $\lambda_{\psi_j^e}(\mf{frob})$ are of the form $q^{ib_j}q^{n/2}$ for some even $n\in \bz$, and the eigenvalues of $\lambda_{\psi_j^o}(\mf{frob})$ are of the form $q^{ib_j}q^{n/2}$ for some odd $n\in \bz$. 
	Moreover, by essentially following the argument in \cite[Lemma 3.1]{Bal}, one can see that
	$$V_{\lambda_\psi}\cong V_{\lambda_{\psi_1^e}}\times V_{\lambda_{\psi_1^o}}\times\cdots \times V_{\lambda_{\psi_r^o}},$$
	with
	 	$$H_{\lambda_\psi}\cong H_{\lambda_{\psi_1^e}}\times H_{\lambda_{\psi_1^o}}\times\cdots \times H_{\lambda_{\psi_r^o}},$$
	 	acting component-wise. 
	 	Thus an orbit $C\cong C_1^e\times C_1^o\times \cdots \times C_r^o$ has smooth closure if and only if every $C_j^e, C_j^o$ has smooth closure. 
	
	For $\varepsilon\in\{e, o\}$, define
	$$D:=\max\{d+a : \norm{\cdot}^{b_ji}\boxtimes S^{d}\boxtimes S^{a} \tx{ is a factor of } \psi_j^\varepsilon \}.$$
	 The eigenvalues of $\lambda_{\psi_j^\varepsilon}(\mf{frob})$ are given by
	  $$q^{b_ji}q^{-D/2}, q^{b_ji}q^{-D/2+1}, \ldots, q^{b_ji}q^{-D/2+k-1}$$
	   for $k:=D+1$.  
	Letting $m_1, \ldots, m_{k}$ denote the multiplicities of the respective eigenvalues, the Vogan variety $V_{\lambda_{\psi_j^\varepsilon}}$ is isomorphic to
	$$\{(x_{1,2}, \ldots, x_{k-1, k})  \mid x_{i, i+1}\in M_{m_i, m_{i+1}}(\bc)\},$$
	which is acted upon by the group $\prod_{i=1}^k\GL_{m_i}(\bc)$ via 
	$$(g_1, \ldots, g_k)\cdot(x_{1,2}, \ldots, x_{k-1,k})=(g_1x_{1,2}g_2^{-1}, g_2x_{2,3}g_3^{-1}, \ldots, g_{k-1}x_{k-1,k}g_k^{-1}).$$

	Define $X_{i,i}:=I_{m_i, m_i}$, and for $i<j$ define
	$$X_{i,j}:=x_{i, i+1}\cdot x_{i+1, i+2}\cdots x_{j-1, j}.$$
By the contiguity theorem of \cite{Zel}, the $H_\lambda$-orbits of $V_\lambda$ are determined by the ranks $r_{i,j}:=\tx{rank} X_{i,j}$, and the closure of the orbit is given by those $(x_{1,2}, \ldots, x_{k-1,k})\in V_\lambda$ for which $\tx{rank} X_{i,j}\leq r_{i, j}$. 
Note that the rank of $X_{i,j}$ is at most 
	$$M_{i,j}:=\min\{m_i, \ldots, m_j\}.$$

	\begin{lemma}\label{lem}
		Let $\lambda$ be an unramified infinitesimal parameter of $\GL_n/F$ for which $\lambda(\mf{frob})$ has eigenvalues $q^{e_1}, \ldots, q^{e_1+k-1}$ for some $e_1\in \bc, k\in \bz^+$.  
		An orbit $C$ in $V_\lambda$ with corresponding tuple $(r_{i,j})$ has smooth closure if and only if for every $i<j$ either there is some $i\leq l<j$ with $r_{l, l+1}=0$, or $r_{i,j}=M_{i,j}$. 
	\end{lemma}
	\begin{proof}
		$(\impliedby)$ Suppose $C$ is an orbit with a rank triangle of above form, and let $i_1, \ldots, i_l$ be the indices for which $r_{i_j, i_j+1}\neq 0$.  
		Then, 
		$$\bar{C}\cong M_{m_{i_1}, m_{i_1+1}}(\bc)\times\cdots \times M_{m_{i_l}, m_{i_l+1}}(\bc)$$
		which is smooth.  
		
		$(\implies)$  We proceed by proving the contrapositive, that is, suppose there is some $i<j$ such that $r_{l, l+1}\neq 0$ for all $i\leq l <j$, and $r_{i,j}\neq M_{i,j}$.   
		Let $[a, b]$ be the largest interval of integers such that $a\leq i, j\leq b$ and $r_{l, l+1}\neq 0$ for all $l\in [a, b-1]$.  
		Then, $\bar{C}$ is the product of the following 3 varieties
		\begin{align*}
			A:=&\{(x_{1,2}, \ldots, x_{a-2, a-1}, 0) : \tx{ rank } X_{i,j}\leq r_{i,j},  1\leq i <j\leq a-1\},\\
			B:=&\{(x_{a, a+1}, \ldots, x_{b-1, b}) : \tx{ rank } X_{i,j}\leq r_{i,j}, a\leq i <j\leq b\},\\
			D:=&\{(0, x_{b+1, b+2}, \ldots, x_{k-1, k}) : \tx{ rank } X_{i,j}\leq r_{i,j}, b+1\leq i <j\leq k\}.
		\end{align*}
		Thus it suffices to prove that $B$ is singular.  
		First note that 
		$$(x_{a, a+1}, \ldots, x_{b-1, b})=(0, \ldots, 0)$$
		is in $B$.  
		 By \cite[Theorem 2.2]{LM}, the equations given by the $(r_{s,t}+1)\times(r_{s,t}+1)$-minors of $X_{s, t}$ when $r_{s, t}\neq M_{s, t}$
		 define the orbit closure scheme-theoretically, and in particular the resulting scheme is reduced.  
		 Therefore it suffices to determine the singular locus using the Jacobian conditions on the defining polynomials.   
		
		By assumption, there is at least one $r_{s, t}\neq M_{s, t}$, and the equations determined by $X_{s, t}$ are homogeneous polynomials of degree $(t-s)(r_{s,t}+1)$, and thus
		$$(x_{a, a+1}, \ldots, x_{b-1, b})=(0, \ldots, 0)$$
		satisfies the equations and all of their partial derivatives so long as none of the polynomials are degree 1.  
		This would only be possible if $t=s+1$ and $r_{s, t}=r_{s, s+1}=0$, contrary to our assumptions (since $a\leq s<t\leq b$).  
		Therefore $B$, hence $\bar{C}$ must be singular.

	\end{proof}
	
	In the following we employ the language of \emph{multisegments} in our proof, and their connection with the ranks $r_{i,j}$ describing a corresponding orbit.    
	While multisegments date back at least to the work of Bernstein-Zelevinsky, we refer the reader to \cite{CR24} for a modern exposition, including those multisegments specifically attached to Arthur parameters.

	\begin{theorem}\label{thm}
		Given a $p$-adic field $F/\bq_p$ and $\psi$ an unramified Arthur parameter of $\GL_n/F$, with decomposition $\psi=\psi_1^e\oplus \psi_1^o\oplus \cdots \oplus \psi_r^o$ as above,
		 $$C_\psi=C_{\psi_1^e}\times C_{\psi_1^o}\times \cdots \times C_{\psi_r^o}$$
		  has smooth closure if and only if each $C_{\psi_j^e}$ (resp. $C_{\psi_j^o}$) is open or closed. 
	\end{theorem}  
	{\allowdisplaybreaks
	\begin{proof}
		First, we suppose that there is only one factor in the decomposition of $\psi$.  
		Thus
		$$\psi=\bigoplus_{j=1}^r\norm{\cdot}^{bi}\boxtimes S^{d_j}\boxtimes S^{a_j},$$
		for some fixed $b\in[0, 2\pi/\log q)$ and either  for all $j$, $(a_j+d_j)$ is even, or for all $j$, $a_j+d_j$ is odd.   
		As varieties $V_{\lambda_\psi}=V_{\lambda_{\norm{\cdot}^{-bi}\psi}}$, and thus it suffices to consider when $b=0$.  
		Let $D:=\max\{a_i+d_i : 1\leq i \leq r\}$.
		The eigenvalues of $\lambda_\psi(\mf{frob})$ are exactly $q^{-D/2 +n}$ for $n=0, \ldots, D$.  
		Writing $e_1:=-D/2$ and $e_i:=-D/2 +i-1$, we note that if for some $j$ for which $a_j+d_j=D$ we have $d_j>0$, then the segments corresponding to $1\boxtimes S^{d_j}\boxtimes S^{a_j}$ have length at least two, and every pair $e_l, e_{l+1}$ is contained in some segment of length at least two.   
		Therefore, in the corresponding orbit, $r_{l, l+1}\geq 1$ for all $1\leq l \leq  D$.  
		If every $r_{i,j}=M_{i,j}$, then $C_\psi$ would be open.  
		Hence, if $C_\psi$ is not open, then by Lemma \ref{lem} it must have singular closure.

		Suppose now that $
		d_i=0$ for every $i$ such that $a_i+d_i=D$.
		Choosing one such $j$, we have $a_j=D$, and the segments corresponding
		to the factor $
		1\boxtimes S^0\boxtimes S^D$
		are precisely the singletons
		\[
		[e_1, e_1],[e_2, e_2],\ldots,[e_{D+1}, e_{D+1}].
		\]
		For every $1\leq l\leq D$, these $e_i$ contribute 1 to the multiplicities $m_i$.  
		Thus, for every $1\leq l \leq D$ we must have
		 $$r_{l, l+1}\leq \min\{m_l-1, m_{l+1}-1\}<\min\{m_l, m_{l+1}\}=M_{l, l+1}.$$
		Hence, by Lemma \ref{lem}, the only case when $\bar{C}_\psi$ is smooth is when $r_{l, l+1}=0$ for all $1\leq l \leq D$, which is exactly when $C_\psi$ is the closed orbit.   
		
		In other words, we have demonstrated that $C_\psi$ is open, closed, or has singular closure.  
		
		{\allowdisplaybreaks
		Now, in the general case
		$$\bar{C}_\psi\cong \bar{C}_{\psi_1^e}\times\cdots \times  \bar{C}_{\psi_r^o}\subseteq V_{\lambda_{\psi_1^e}}\times\cdots \times  V_{\lambda_{\psi_r^o}}\cong V_{\lambda_\psi}$$
		is smooth if and only if each $\bar{C}_{\psi_j^e}$ and $\bar{C}_{\psi_j^o}$ are both smooth, and by the above this only occurs when each $C_{\psi_j^e}$ is open or closed, and $C_{\psi_j^o}$ is open or closed. 
	}
	\end{proof}
}
	
	{\allowdisplaybreaks
	The above argument should carry over to each of the other classical groups as well.  
	The only barrier is that at the time of writing there is no result akin to \cite[Theorem 2.2]{LM} giving a set of equations describing the orbit closures as a reduced scheme.  
	Most likely, it is the case, but it remains unproven.  
	
	Since Speh representations of $\GL_n(F)$ are of Arthur type with irreducible Arthur parameter, we have the following immediate Corollary relating the geometry to representation-theoretic objects. 
	
	\begin{corollary}
		For the Langlands parameter $\phi$ of a Speh representation of $\GL_n(F)$, $C_\phi$ is open, closed, or has singular closure. 
	\end{corollary}
	
	Given the enumeration of Vogan varieties for unramified infinitesimal parameters of $G_2$ in \cite{BCFZ}, one can verify by inspection that the analogous result to Theorem \ref{thm} holds for $G_2$.  
	In particular, the only two Arthur orbits which are not open, closed, or have singular closure belong to the only factorizable Vogan variety appearing in \cite[Example 2.10]{BCFZ}.
}

\textbf{Statement of AI use}

All the theorems, ideas, proofs, and text of this document are completely human-generated.  
Chat GPT 5.6 Sol suggested the reference \cite{LM} for the reducedness of the standard equations defining the orbit closures, and was used to find spelling, grammatical, and minor mathematical errors.

\end{document}